\documentclass{amsart}
\usepackage{amssymb,amsmath,latexsym}
\usepackage{amsthm}
\usepackage{fontenc}
\usepackage{amssymb}

\numberwithin{equation}{section}

\newtheorem{theorem}{Theorem}[section]

\begin{document}
\author{Alexander E. Patkowski}
\title{On curious properties of the general size-biased distribution}

\maketitle

\begin{abstract} We offer further results on a general size-biased distribution related to the Riemann xi-function we presented in [9] using the work of Ferrar. Curious properties associated with its expected value are presented, which are related to special functional equations. We also relate our observations to some recent developments related to the Riemann hypothesis.\end{abstract}

\keywords{\it Keywords: \rm Size-biased distribution; Riemann xi-function; Probability}

\subjclass{ \it 2020 Mathematics Subject Classification 11M06, 60E07.}

\section{Introduction} 
Let  $\zeta(s)=\sum_{n=1}^{\infty}n^{-s}$ be the Riemann zeta function for $\Re(s)>1,$ and $\Gamma(s)$ the classical gamma function. The Riemann xi-function is given as [11, pg.16, eq.(2.1.12)] $\xi(s):=\frac{1}{2}s(s-1)\pi^{-\frac{s}{2}}\Gamma(\frac{s}{2})\zeta(s).$ The functional equation for this function is [11, pg.16, eq.(2.1.13)] $\xi(s)=\xi(1-s),$ which may be utilized to establish an example of the functional relationship studied by Ferrar [6]. In [9] we defined the density function $x^{-1}v_k(x),$ for $0<c<1$
\begin{equation} v_k(x)=\frac{1}{2\pi i}\int_{(c)}2^{k}\xi^k(s)x^{s-1}ds.\end{equation} To see  [2, p. 34, F2], note that $\xi(0)^k=\xi^k(1)=2^{-k}$ (e.g. [1, pg.439]). In the case of $k=1$ we obtain the familiar Riemann xi integral from [5, pg.207--208], [5, pg.228], or the probabilistic density function contained in [1],

\begin{equation}\int_{0}^{\infty}x^{s-1}\Theta(x)dx=\xi(s),\end{equation}
where 
\begin{equation}\Theta(x):=2x^2\sum_{n\ge1}(2\pi^2 n^4x^2-3\pi n^2)e^{-\pi n^2 x^2}.\end{equation} 

\par The functional equation $x\Theta(x)=\Theta(x^{-1})$ translates to the Riemann xi functional equation when applying Mellin transforms. This underlying observation was key in establishing the sized-biased distribution property (see [1, pg.439]) in the case $k=1.$  In other words, we have the expected value ([2, pg.65--66]) $\mathop{\mathbb{E}(X_1f(X_1))}=\mathop{\mathbb{E}(f(\frac{1}{X_1}))},$ for any measurable function $f,$ for random variable $X_1$ associated with this distribution.  If we denote the random variable $X_k$ corresponding to the density function $x^{-1}v_k(x),$ then we have [9, Theorem 1.1] the size biased property $$\mathop{\mathbb{E}\left(X_kf(X_k)\right)}=\mathop{\mathbb{E}\left(f(\frac{1}{X_k})\right)}.$$ Iterating Parseval's theorem for Mellin transforms [10 pg.83, eq.(3.1.12)], we established in [9] that $v_k(x)>0$ for $k\ge1.$ Furthmore, in [9, Theorem 1.2] we also noted that $X_{k-1}$ majorizes $X_{k}$ in distribution. Since we only outlined a possible proof therein, we will give details here to establish this result. Following Edwards [5, pg.228], we write (1.1) as 
\begin{equation} v_k(x)=\frac{1}{\pi}\int_{-\infty}^{\infty}2^{k}\xi^k(\frac{1}{2}+it)x^{it-1/2}dt.\end{equation}
We use the estimate from Titchmarsh [11, pg.257] \begin{equation}\xi(\frac{1}{2}+it)=O(|t|^Ae^{-\pi |t|/4}),\end{equation} $A>0,$ as $t\rightarrow\pm\infty,$ to justify interchange of integration to obtain the power series from (1.4) as 
\begin{equation}2^{-k}x^{1/2}v_k(x)=\sum_{n\ge0}c_{n,k}(i\log(x))^{n},\end{equation}
where
$$c_{n,k}=\frac{1}{\pi n!}\int_{-\infty}^{\infty}\xi(\frac{1}{2}+it)^kt^{n}dt.$$ To establish that $P(X_{k}\ge x)\le P(X_{k-1}\ge x)$ it is enough to establish that $v_{k-1}(x)C\ge v_k(x),$ $x>\delta,$ for some $\delta>0,$ and a positive constant $C,$ or equivalently $c_{n,k-1}C\ge c_{n,k}.$ First, we observe that applying the estimate (1.5)
$$\begin{aligned} &c_{n,k}\le \frac{C}{\pi n!}\int_{0}^{\infty}|t|^{n+Ak}e^{-\pi k|t|/4}dt \\
&=\frac{C\Gamma(n+Ak+1)}{\pi n!}(\frac{k\pi}{4})^{-(n+Ak+1)}, \end{aligned}$$
for a constant $C.$ Now the desired inequality will be established if 
$$\frac{\Gamma(n+A(k-1)+1)}{\pi n!}\left(\frac{(k-1)\pi}{4}\right)^{-(n+A(k-1)+1)}\ge\frac{\Gamma(n+Ak+1)}{\pi n!}\left(\frac{k\pi}{4}\right)^{-(n+Ak+1)},$$ or equivalently,
$$\left(k-1\right)^{-(n+A(k-1)+1)}(\frac{\pi}{4})^A\ge(n+Ak)\cdots(n+Ak-A+1)\left(k\right)^{-(n+Ak+1)}.$$
This inequality is true for all integers $n, k\ge2,$ since $(n+Ak)\cdots(n+Ak-A+1)$ may be viewed as polynomial in $k$ with degree $A\ge1,$ and hence the result follows. Another way one could establish the inequality follows from observing the function $$\frac{\Gamma(n+Ak+1)}{\pi n!}\left(\frac{k\pi}{4}\right)^{-(n+Ak+1)}$$ decreases as $k$ increases. This follows from taking the logarithmic derivative of the gamma function, observing the growth [10, pg.134] of $\frac{\partial}{\partial k}\log\left(\Gamma(k+1)\right)$ is $\log(k),$ and applying the first derivative test. We can also establish this property using another curious result which appears to follow from [10, pg.83, eq.(3.1.13)], which we state in the following theorem.
\begin{theorem} If $X_k$ is the general size-biased distribution corresponding to the density function  $x^{-1}v_k(x),$ then for integers $m\ge1,$ $z>0,$
$$\mathop{\mathbb{E}(X_kv_m(X_kz))}=v_{m+k}(z).$$
\end{theorem} 
\begin{proof}
By [10, pg.83, eq.(3.1.12)],
$$\begin{aligned} &v_{m+k}(z)=\frac{1}{2\pi i}\int_{(c)}2^{m+k}\xi^{m+k}(s)z^{s-1}ds \\
&=\frac{1}{2\pi i}\int_{(c)}2^{m+k}\xi^{m}(s)\xi^{k}(s)z^{s-1}ds\\
&=\int_{0}^{\infty} v_k(x)v_m(xz)dx\\
&=\mathop{\mathbb{E}(X_kv_m(X_kz))}. \end{aligned}$$
\end{proof}
If we apply Theorem 1.1 with $m=1,$ and $v_1(x)=\Theta(x)=O(x^Ae^{-x^2}),$ we see that
$$\begin{aligned} &v_k(z)=\int_{0}^{\infty} v_{k-1}(x)v_1(xz)dx \\
&=\int_{0}^{\infty} v_1(x)v_{k-1}(xz)dx\\
&=O\left(\int_{0}^{\infty} v_{k-1}(xz)x^Ae^{-x^2}dx\right). \end{aligned}$$
Since $v_k(x)\rightarrow0,$ when $x\rightarrow\infty,$ $v_k(x_1)\le v_k(x_2),$ when $x_1\ge x_2>\delta$ for some $\delta>0.$ That is, $v_k(x)$ is monotonically decreasing from [10, pg.80, eq.(3.1.3)] and absolute convergence of (1.1). Therefore, $v_k(x_1x_2)\le v_k(x_2),$ Consequently,
$$\int_{0}^{\infty} v_{k-1}(xz)v_1(x)dx=O\left(\int_{0}^{\infty} v_{k-1}(xz)x^Ae^{-x^2}dx\right)=O( v_{k-1}(z)).$$ Now applying the positivity of $v_k(x)$ gives $v_k(x)\le Cv_{k-1}(x),$ $x>\delta,$ for some $\delta>0$ and a positive constant $C,$ and again gives the majorizing property. 
\par An interesting remark regarding Theorem 1.1 should be noted. In particular, the only distribution we are aware that has the expected vaue of its density function equal to its density at any point is the uniform distribution. That is if $U$ is a uniformly distributed random variable with density function $u(x)=(b-a)^{-1},$ on $[a,b],$ then $\mathop{\mathbb{E}(u(Uz))}=u(z).$ Hence in this sense, Theorem 1.1 appears to be a uniform-like property, but with the obvious difference that its value is not constant across support. 

\section{A functional relation for the general size-biased distribution}

In this section we will utilize the main result of our note [8] to obtain a general relation for expected values of our general size-biased distribution. Let $g(z)$ be an analytic function for $\Re(z)>0$ which is entire of order one. Let $H(t)$ be a locally integratable function which satisfies the condition $H(t)=t^{-1}H(t^{-1}),$ and has the Mellin transform $g(s)$ for $s\in\mathbb{C}.$ Then by [8, Proposition 1.1] we have that 
$$f(z,y)=\frac{1}{x}\int_{0}^{\infty}\frac{zt}{zt+1}t^{-y/x+1/x-1}H(t^{1/x})dt,$$ is a solution to the functional equation,
\begin{equation} f(z,y+x)+zf(z,y)=zg(y).\end{equation} Here $0<x<\pi/r$ such that for small $\delta>0,$
\begin{equation}|g(s)|<Ce^{(r-\delta)|s|},\end{equation}

\begin{theorem} Let $x$ be chosen to satisfy (2.2) in the case of $g(s)=\xi^{k}(s).$ If $X_k$ is the general size-biased distribution corresponding to the density function  $x^{-1}v_k(x),$ then for $y, z>0,$ then
$$\mathop{\mathbb{E}\left(\frac{zX_k^{1-y}}{zX_k^x+1}\right)}+z\mathop{\mathbb{E}\left(\frac{zX_k^{1+x-y}}{zX_k^x+1}\right)}=z\mathop{\mathbb{E}\left(X_k^{y}\right)}=z\xi^k(y).$$
\end{theorem}
\begin{proof} Note that since $\xi(s)$ is entire of order one, so is $\xi(s)^k$ [11, pg. 29, Theorem 2.12]. Choosing $H(t)=v_k(t),$ and a change of variables tells us that
$$\begin{aligned}&\frac{1}{x}\int_{0}^{\infty}\frac{zt}{zt+1}t^{-y/x+1/x-1}H(t^{1/x})dt\\ 
&=\int_{0}^{\infty}\frac{zt^x}{zt^x+1}t^{-y}v_k(t)dt\\
&=\mathop{\mathbb{E}\left(\frac{zX_k^{1+x-y}}{zX_k^x+1}\right)}. \end{aligned}$$
The result now follows from (2.1).
\end{proof}
Note that since [11, pg.255] the Riemann hypothesis is equivalent to $\xi(\frac{1}{2}+it)$ having only real zeros, the Riemann hypothesis is equivalent to
$$\mathop{\mathbb{E}\left(\frac{zX_k^{1/2-iy}}{zX_k^x+1}\right)}+z\mathop{\mathbb{E}\left(\frac{zX_k^{1/2+x-iy}}{zX_k^x+1}\right)}=0$$
for real $y$ only.

\section{A relationship with the Riemann hypothesis}

In [7, eq.(1.1)] we replace $x$ with $x^2$ to find
\begin{equation}\Xi_{\lambda}(t)=2\int_{0}^{\infty}e^{\lambda(\log(x))^2+it\log(x)}\sqrt{x}\Theta(x)\frac{dx}{x}=2\int_{-\infty}^{\infty}e^{\lambda u^2}e^{u/2}\Theta(e^{u})e^{itu}du.\end{equation}
The integral on the right side of (3.1) has been presented in [4] in an effort to approach the Riemann hypothesis. It is known in [7, pg.282] that $\Xi_{\lambda}(t)$ has only real zeros for $\lambda\ge 1/8,$ and the Riemann hypothesis is equivalent to $\Xi_{0}(t)$ having only real zeros  (using (1.2) and [11, pg.255], for example). The approach uses [4, Theorem 13], which in short states that if a Fourier transform 
$$F(t)=\int_{-\infty}^{\infty}f(u)e^{itu}du,$$
satifies certain growth conditions (see [4, Theorem 10]) with roots in the strip $|\Im(t)|\le\Delta,$ then the function
$$\int_{-\infty}^{\infty}e^{\frac{\lambda^2}{2} u^2}f(u)e^{itu}du,$$
has roots in the strip $|\Im(t)|\le\sqrt{\max{(\Delta^2-\lambda^2,0)}}.$ We are able to extend the results in [7] for our distribution. Recalling our previous observation that $v_k(x)\le Cv_{k-1}(x)\le Cv_1(x),$ $x>\delta,$ for some $\delta>0,$ we can see by taking expected values,
$$\mathop{\mathbb{E}\left(\sqrt{X_k}e^{\lambda(\log(X_k))^2+it\log(X_k)}\right)}\le C\mathop{\mathbb{E}\left(\sqrt{X_1}e^{\lambda(\log(X_1))^2+it\log(X_1)}\right)}.$$
First note that both sides of this inequality represent an entire function. It would follow naturally from Liouville's theorem that should the right side be zero for infinitely many $t$ so should the left side.

\begin{theorem} If $X_k$ is the general size-biased distribution corresponding to the density function  $x^{-1}v_k(x),$ then
\begin{equation}\mathop{\mathbb{E}\left(\sqrt{X_k}e^{\lambda(\log(X_k))^2+it\log(X_k)}\right)}=0,\end{equation}
is true for only real $t$ when $\lambda\ge1/8.$ Furthermore, (3.2) is true when $\lambda=0$ infinitely often.
\end{theorem}

\begin{proof} First note that $e^{u/2}v_k(e^{u})$ represents an even function which tends to $0$ as $u\rightarrow\pm\infty$ thereby satisfying [4, Theorem 10]. Using the same arguments as those in [7, pg.282] it is seen that for $k\ge1$ the more general function,
\begin{equation}\Xi_{\lambda,k}(t):=2\int_{0}^{\infty}e^{\lambda(\log(x))^2+it\log(x)}\sqrt{x}v_k(x)\frac{dx}{x}=2\int_{-\infty}^{\infty}e^{\lambda u^2}e^{u/2}v_k(e^{u})e^{itu}du,\end{equation}
is an even entire function of order one, and has infinitely many real zeros. By the Paley-Wiener theorem, (3.3) cannot be of finite exponential type. An entire function with finite order and maximal type and has infinitely many zeros by Hadamard's factorization theorem. Therefore, [7, Proposition A] holds true for the function (3.3) as well. [7, Proposition A] says that if $\lambda_1\le\lambda_2,$ $\Delta\ge0,$ and the zeros of $\Xi_{\lambda_1}$ lie in $\{t: |\Im(t)|\le\Delta\},$ then those of  $\Xi_{\lambda_2}$ lie in $\{t: |\Im(t)|\le\bar{\Delta}\},$  where $\bar{\Delta}=\sqrt{\max{(\Delta^2-2(\lambda_2-\lambda_1,0)}}.$ We put $\lambda_1=0,$ and note that $\Xi_0(t)$ roots lie in $\{t: |\Im(t)|\le1/2\},$ by [5, pg.301], and $\Xi_{0,k}(t)$ are the same roots. Hence $\sqrt{\max{(1/4-2\lambda_2,0)}}$ gives $\lambda\ge1/8.$ The second statement in the theorem follows from Hardy's theorem [11, pg.256--257].
\end{proof}
It is now well-known that (3.1) satisfies the backward heat equation, and so similarly for (3.3) it is straightforward to see,
\begin{equation}\frac{\partial^2 \Xi_{\lambda,k}(t)}{\partial t^2}=-\frac{\partial \Xi_{\lambda,k}(t)}{\partial \lambda}.\end{equation}
The additional boundary condition $ \Xi_{0,k}(t)=0,$ only if $t\in\mathbb{R}$ would be equivalent to the Riemann hypothesis. It is known that the backward heat equation and Jacobi theta function have a relationship to Brownian motion [3]. The Jacobi theta function has also appeared in the context of diffusion processes as a solution to another boundary value problem in probability theory [2, Theorem 21.6.1], [2, pg.609]. However, the fundamental solution differs from (3.1), as the construction of $\Xi_{\lambda,k}(t)$ has come about in a superficial manner in applying a theorem from analysis concerning zeros of certain Fourier transforms. In fact, the backward heat equation is not well posed. It would be interesting to establish a detailed comparison with the fundamental solution.

1390 Bumps River Rd. \\*
Centerville, MA
02632 \\*
USA \\*

ul. A. E. Ody\'{n}ca 47 \\*
02-606 Warsaw\\*
Poland\\*
E-mail: alexpatk@hotmail.com, alexepatkowski@gmail.com
 \\*
Competing interests: The author declares none.
 \\*
Funding statement: The author did not receive support for the submitted work.


\begin{thebibliography}{9}

\bibitem{ConcreteMath} P. Biane, J. Pitman and M. Yor, \emph{Probability laws related to the Jacobi theta and Riemann zeta
functions, and Brownian excursions,} Bull. AMS, 4 (2001), 435--465.

\bibitem{ConcreteMath} A.A. Borovkov. \emph{Probability Theory.} Universitext. Springer, London, 2013.

\bibitem{ConcreteMath} K.L. Chung. Excursions in Brownian motion. Arkiv fur Matematik, 14:155--177, 1976.

\bibitem{ConcreteMath} N.G. de Bruijn, \emph{The roots of trigonometric integrals,} Duke Math. J. 17 (1950) 197–226.

\bibitem{ConcreteMath} H. M. Edwards. \emph{Riemann's Zeta Function,} 1974. Dover Publications.


\bibitem{ConcreteMath} W. L. Ferrar, \emph{Some solutions of the equation $F(t) = F(t^{-1})$,} J. London Math. Soc. 11 (1936), 99--103.

\bibitem{ConcreteMath} Ki, H., Kim, Y.-O., Lee, J. \emph{On the de Bruijn–Newman constant,} Adv. Math. 222, 281–306 (2009)

\bibitem{ConcreteMath} A. E. Patkowski, \emph{On a solution to a functional equation,} Journal of Applied Analysis, vol. 28, no. 1, 2022, pp. 91-93.

\bibitem{ConcreteMath} A. E. Patkowski, \emph{A general size-biased distribution,} Journal of Applied Analysis, vol. 29, no. 2, 2023, pp. 347-351.

\bibitem{ConcreteMath} R. B. Paris, D. Kaminski, \emph{ Asymptotics and Mellin--Barnes Integrals.} Cambridge University Press. (2001)



\bibitem{ConcereteMath} E. C. Titchmarsh, \emph{The theory of the Riemann zeta function,} Oxford University Press,
2nd edition, 1986.



\end{thebibliography}
\end{document}